\documentclass[amscd,amssymb,verbatim,12pt]{amsart}
\usepackage{graphicx}

\theoremstyle{definition}

\newif\ifAMS
\IfFileExists{amssymb.sty}
  {\AMStrue\usepackage{amssymb}}
  {\usepackage{latexsym}}
\theoremstyle{plain}

\newtheorem{Thm}{Theorem}[section]
\newtheorem{Cor}[Thm]{Corollary}
\newtheorem{Lem}[Thm]{Lemma}

\theoremstyle{definition}
\newtheorem{Def}[Thm]{Definition}

\theoremstyle{remark}

\newtheorem{Ex}[Thm]{Example}

\newcommand{\interior}{^{ \kern-5pt ^\circ}}
\newcommand {\bd}{\partial}

\newcommand {\R}{{\mathbb R}}
\newcommand {\Z}{{\mathbb Z}}

\newcommand {\E}{{\mathbb E}}

\begin{document}
\title{Asymptotic dimension of 3-dimensional CAT(0) manifolds}

\author{ Panos Papasoglu}

\author
{Eric Swenson }

\email{papazoglou@maths.ox.ac.uk}
\address{Mathematical Institute, University of Oxford,\\ Woodstock Road, Oxford OX2 6GG, U.K.}

\email {eric@math.byu.edu}
\address
 {Mathematics Department, Brigham Young University,
Provo UT 84602}
\begin{abstract}
We prove that every $3$-manifold equipped with a proper metric
admitting a convex geodesic bicombing has Assouad--Nagata dimension
at most $3$. In particular, every proper CAT$(0)$ $3$-manifold has
Assouad--Nagata dimension at most $3$.  As a corollary, the isoperimetric results of Lang--Stadler--Urech
\cite{LSU} and Peteranderl \cite{Pet}, previously applicable to
$3$-dimensional Hadamard manifolds, also apply to proper CAT$(0)$
$3$-manifolds of asymptotic rank at most $2$.
\end{abstract}
\thanks{}
\maketitle

\section {Introduction}

The notion of asymptotic dimension was introduced by Gromov and has
become an important invariant in geometric group theory and coarse
geometry.  The Assouad--Nagata dimension is a closely related metric
notion in which one requires linear control on the diameter of the
sets appearing in the covers.  In particular,
\[
\operatorname{asdim} X\leq
\operatorname{asdim}_{\mathrm{AN}} X\leq
\dim_{\mathrm{AN}} X.
\]

In \cite{FP} it was shown that a geodesic metric space which admits
an injective continuous map into $\R^2$ has Assouad--Nagata dimension
at most three, uniformly.  J{\o}rgensen and Lang \cite{JL} improved
this bound to two (see also \cite{BBEGLPS}).  They also used this result, together with a
Hurewicz type argument, to show that every $3$-dimensional Hadamard
manifold has Assouad--Nagata dimension three.

In this paper we consider the corresponding question without the
smoothness assumption.  In fact we work in the more general setting
of metric manifolds admitting convex geodesic bicombings.  Recall that
CAT(0) spaces, Busemann spaces and Hadamard manifolds admit convex
geodesic bicombings; see \cite{DE-LA}.  Our main result is the
following:

\medskip
\noindent
{\bf Theorem \ref{thm:AN-combing}.}
{\em Let $X$ be a $3$-manifold with a proper metric.  If $X$ admits
a convex geodesic bicombing, then
\[
\dim_{\mathrm{AN}}(X)\leq 3.
\]
In particular,
\[
\operatorname{asdim}_{\mathrm{AN}}(X)\leq 3
\quad\text{and}\quad
\operatorname{asdim}(X)\leq 3.
\]}
\medskip

Thus, in particular, every proper CAT(0) $3$-manifold has
Assouad--Nagata dimension at most three.

One motivation for this result comes from the recent work of
Lang--Stadler--Urech \cite{LSU}. They prove an isoperimetric gap
theorem for CAT(0) spaces of asymptotic rank at most two and finite
asymptotic Assouad--Nagata dimension, with filling exponent
arbitrarily close to $1$. A recent theorem of Peteranderl \cite{Pet}
strengthens this to the optimal linear filling inequality. Their
results apply, in particular, to $3$-dimensional Hadamard manifolds
by the theorem of J{\o}rgensen--Lang. Theorem~\ref{thm:AN-combing}
allows one to replace Hadamard manifolds by CAT(0) $3$-manifolds.

We use the language of metric integral currents in the sense of
Ambrosio--Kirchheim \cite{AK}. Informally, an element of
${\mathbf I}_{k,\mathrm c}(X)$ may be viewed as a compactly supported,
oriented $k$-dimensional Lipschitz chain, allowing integer
multiplicities and suitable limits. Its boundary is denoted by
$\partial T$, and its mass ${\mathbf M}(T)$ is the metric analogue of
$k$-dimensional volume. A current $T$ is a cycle if $\partial T=0$,
and a filling of $T$ is a current $V$ with $\partial V=T$. We use the
notion of asymptotic rank as in \cite{LSU}.

\medskip
\noindent
{\bf Corollary \ref{cor:isoperimetric}.}
{\em Let $X$ be a proper CAT$(0)$ $3$-manifold of asymptotic rank at
most $2$. Then, for every $k\geq2$, there is a constant $C$ such that
every cycle $T\in {\mathbf I}_{k,\mathrm c}(X)$ admits a filling
$V\in {\mathbf I}_{k+1,\mathrm c}(X)$ satisfying
\[
\partial V=T
\qquad\text{and}\qquad
{\mathbf M}(V)\leq C\,{\mathbf M}(T).
\]}
\medskip

\subsection{Proof outline}
We briefly describe the argument.  Fix a point $p\in X$.  The first
part of the paper is concerned with the metric spheres
\[
S(p,r)=\{x\in X:d(p,x)=r\}.
\]
The convex geodesic bicombing gives a contraction of metric balls
towards $p$.  Using this contraction and results on homology
manifolds, we show that, when $X$ is a $3$-manifold, every
$S(p,r)$ is a topological $2$-sphere.  We further show that the path
metric on $S(p,r)$ induces its original topology.  In particular,
with its path metric, $S(p,r)$ is a compact geodesic metric
$2$-sphere.

We then use the arguments of Fujiwara--Papasoglu \cite{FP} and
J{\o}rgensen--Lang \cite{JL}. Although their results are stated for
planar geodesic spaces, the topological use of planarity in the
relevant argument is only the separation property supplied by the
Jordan curve theorem, and the same argument applies to a geodesic
metric $2$-sphere. It follows that the metric spheres $S(p,r)$ admit
uniformly controlled three-colour covers at every scale.

The next step is to pass from spheres to annuli.  If
\[
A(p;r,s)=\{x\in X:r\leq d(p,x)<r+s\},
\]
the bicombing gives a radial map from a slightly enlarged annulus to
a sphere lying just inside it.  This map is $2$-Lipschitz, so the
cover of the sphere gives a uniformly bounded cover of the annulus.
There is a small point here concerning the two metrics on the
annulus.  The separation of the resulting cover is initially obtained
for the path metric of the enlarged annulus, whereas the
Assouad--Nagata dimension is defined using the metric of $X$.  If
$x,y$ lie in the original annulus and $d(x,y)<s$, convexity of the
bicombing implies that the geodesic $\sigma_{xy}$ stays inside the
enlarged annulus.  Consequently the path metric and the metric of
$X$ agree for such $x,y$.  This gives the required separation in the
original metric.

Finally, a Hurewicz type covering argument, adapted from
\cite{JL}, combines the uniformly controlled covers of the annuli.
This gives a four-colour cover of $X$ at every scale and proves
Theorem~\ref{thm:AN-combing}. The linear form of
Corollary~\ref{cor:isoperimetric} then follows from \cite{Pet}; the
earlier estimate with exponent $1+\delta$ follows from
\cite[Theorem~1.1]{LSU}.

\section{Preliminaries}
For $X$ a metric space with metric $d$, $p\in X$, and $r>0$ we define the open metric ball
\[
B(p,r)=\{x\in X:\ d(p,x)<r\},
\]
and the closed metric ball
\[
\bar B(p,r)=\{x\in X:\ d(p,x)\leq r\}.
\]
The sphere of radius $r$ about $p$ is
\[
S(p,r)=\{x\in X:\ d(p,x)=r\}=\bar B(p,r)\setminus B(p,r).
\]

We refer the reader to \cite{BRI-HAE} or \cite{BAL} for more details of the following.
A geodesic in a metric space $X$ is an isometric embedding of an interval of $\R$.
\begin{Def}
For $X$ a geodesic metric space and $\Delta(a,b,c)$ a geodesic triangle in $X$ with vertices $a,b,c\in X$, let $\bar\Delta=\Delta(\bar a,\bar b,\bar c)\subset \E^2$ be a comparison triangle, so that
\[
d(a,b)=d(\bar a,\bar b),\qquad d(a,c)=d(\bar a,\bar c),\qquad d(b,c)=d(\bar b,\bar c).
\]
For a point $z$ on a specified side of $\Delta(a,b,c)$, its comparison point $\bar z$ is the point on the corresponding side of $\bar\Delta$ at the same distance from the corresponding endpoint. We say that $\Delta(a,b,c)$ is CAT(0) if, for any points $y,z$ on specified sides of $\Delta(a,b,c)$ and their corresponding comparison points $\bar y,\bar z\in\bar\Delta$,
\[
d(y,z)\leq d(\bar y,\bar z).
\]
The space $X$ is said to be CAT(0) if every geodesic triangle in $X$ is CAT(0).
\end{Def}

\section{Convex geodesic bicombings} The following idea is from \cite{DE-LA}.
\begin{Def} A {\em geodesic bicombing} is a function $\sigma: X\times X \times[0,1] \to X$ such that for any $x,y \in X$, $\sigma_{xy}(t):= \sigma(x,y,t)$ defines a constant speed geodesic $\sigma_{xy}:[0,1] \to X$ from $x$ to $y$. We say that $\sigma$ is {\em convex} if $t \mapsto d(\sigma_{ab}(t), \sigma_{cd}(t))$ is a convex function for any $a,b,c,d \in X$.  We say that a geodesic bicombing is {\em consistent} if whenever $p =\sigma_{xy}(s)$ and $q=\sigma_{xy}(t)$ for $0\le s<t \le 1$ then $\sigma_{pq}$ is just $\sigma_{xy}|_{[s,t]}$ re-parameterized.  This is a possible issue since geodesics might not be unique.
\end{Def}
Hadamard manifolds, CAT(0) spaces and Busemann spaces are all examples of spaces which admit a convex geodesic bicombing. 
\begin{Ex}
In \cite{DE-LA} it is stated that every normed linear space admits a convex geodesic bicombing.  We will explain this here. Let $V$ be a normed linear space; the line segment from $\mathbf u $ to $\mathbf v$ in $V$ is $\sigma_{\mathbf u \mathbf v}(t) =  (1-t)\mathbf u +t 
\mathbf v$, $0\le t\le 1$.  To see that this is a constant speed geodesic, notice that 
$\sigma_{\mathbf 0 \mathbf v}(t) = t 
\mathbf v$ and so $||\sigma_{\mathbf 0 \mathbf v}(t)||= t|| \mathbf v||$.  Thus $\sigma_{\mathbf 0 \mathbf v}$ is a constant speed geodesic for any $\mathbf v \in V$ and since translations are isometries of $V$, $\sigma_{\mathbf u \mathbf v}$ is a constant speed geodesic of $V$ for any $\mathbf u, \mathbf v \in V$.  This now defines our consistent geodesic bicombing $\sigma: V \times V \times [0,1]\to V$.   

Fix $\mathbf a, \mathbf b, \mathbf c,\mathbf d \in V$ and consider the function $f(t) = d\left(\sigma_{\mathbf a \mathbf b}(t), \sigma_{\mathbf c \mathbf d}(t)\right)$
\begin{align}f(t) &= ||[(1-t)\mathbf a +t\mathbf b] -[(1-t)\mathbf c +t\mathbf d]|| \notag\\
&=||(1-t)(\mathbf a-\mathbf c) +t(\mathbf b-\mathbf d)|| \notag\\
&\le (1-t)||\mathbf a-\mathbf c|| +t ||\mathbf b-\mathbf d|| \notag\\
&= (1-t)f(0) + t f(1) \notag\end{align}
Since $\sigma$ is consistent it follows  that $f$ is a convex function (see \cite{DE-LA}).
\end{Ex}
We have the following easy result:
%The following obvious result was not even mentioned in \cite{DE-LA}.
\begin{Lem}
    A convex geodesic bicombing is continuous.
\end{Lem}
\begin{proof}
    Clearly $X\times X \times [0,1] $ is metrizable.  Let $(x_n,y_n,t_n)$ be a sequence in $X\times X\times [0,1]$ with $x_n \to x\in X$, $y_n \to y\in X$, and $t_n \to t \in [0,1]$. Notice that $$d(\sigma(x,y,t),\sigma(x_n,y_n,t_n))\le d(\sigma_{xy}(t),\sigma_{xy}(t_n)) +d(\sigma_{xy}(t_n),\sigma_{x_ny_n}(t_n)).$$
    By convexity $d(\sigma_{xy}(t_n), \sigma_{x_ny_n}(t_n))\le (1-t_n)d(x,x_n) +t_nd(y,y_n)$.
    Since $d(x,x_n)$, $d(y,y_n)$ and $d(\sigma_{xy}(t),\sigma_{xy}(t_n))$ all go to $0$ the result follows.
\end{proof}
\begin{Def}For $X$ a metric space with a convex geodesic bicombing, and $p \in X$, we define the geodesic contraction to $p$, $\tau_p:X\times[0,1]\to X$, by $\tau_p(y,t)=\sigma(p,y,t)$. 
\end{Def} Since $\tau_p$ is just the restriction of $\sigma$ to the subspace $\{p\}\times X \times [0,1]$, $\tau_p$ is continuous. Thus a metric space with a convex geodesic bicombing is contractible.
\begin{Lem} \label{C:cont}
    If $X$ is a metric space with a convex geodesic bicombing, the contraction $\tau_p$ defined above is distance non-increasing and for any $A$ with
   $B(p,r)\subset A \subset \bar B(p,r)$ for $r >0$, the restriction $\tau_p|_A$ gives a distance non-increasing contraction of $A$ to $p$, with $\tau_p(A\times [0,1)) \subset B(p,r)$
\end{Lem}
\begin{proof}
First we show that $\tau$ is distance non-increasing. Let $x,y \in X$ and $t \in [0,1]$.
By convexity,
 $$d(\tau(x,t), \tau(y,t))= d(\sigma_{px}(t),\sigma_{py}(t)) \le td(x,y)\le d(x,y).$$
 
For any $y\in A$, $\tau(y,1) = \sigma(p,y,1) =y$, and $\tau(y,0) = \sigma(p,y,0) =p$.  
For $y\in A$ and $t<1$, the point $\tau(y,t)$ is on a geodesic from $p$ to $y$ but is not equal to $y$.  It follows that $d(p, \tau(y,t)) < d(p,y) \le r$.  Thus $\tau (A \times [0,1)) \subset B(p,r)$.

\end{proof}
\begin{Lem}\label{L:lip}
  If $X$ is a metric space with a convex geodesic bicombing $\sigma$, then for any $p \in X$ and any bounded $A \subset X \setminus B(p,r)$  (for some $r>0$) $\tau_p$ induces a homotopy of $A$ into $S(p,r)$ in $X\setminus B(p,r)$. The radial map $H(\cdot,0):A\to S(p,r)$ is $2$-Lipschitz.
\end{Lem}
\begin{proof}
    The function $g:X \to \R$  defined by      $g(x) =d(p,x)$ is continuous.  Since the geodesics of $\sigma$ are constant speed, for any $x\in X\setminus B(p,r)$, $\tau_p\left(x, \frac r{g(x)}\right) \in S(p,r)$. Define the homotopy $H: A\times [0,1] \to X $ by $$H(a,t) = \tau_p\left(a, \left(1-\frac r{g(a)}\right)t +\frac r {g(a)} \right).$$  $H$ is clearly continuous.  For any $t \in [0,1]$, $\left(1-\frac r{g(a)}\right)t +\frac r {g(a)} \ge \frac r{g(a)}$ so $H(a,t) \not \in B(p,r)$.

    We now show that $H(\cdot,0)$ is $2$-Lipschitz.  Let $x,y \in A$ with $g(x) \le g(y)$.  Let $\hat t =\frac r {g(x)}$ so that  $\hat x = \tau_p (x,\hat t) =H(x,0)\in S(p,r)$, and set $\hat y =\tau_p(y, \hat t)$.
    Thus $\hat y \not \in B(p,r)$ and $d( \hat y,\hat x) \le d(y,x)$ since $\tau_p$ is distance non-increasing. The image of $y$ under $H(\_, 0)$, $\tilde y$, lies on a geodesic from $\hat y$ to $p$.  It follows that $$d(\hat y, \tilde y) = d(\hat  y, S(p,r)) \le   d(\hat y, \hat x) \le d(y,x)$$ 
    Thus $d( \tilde y, \hat x) \le d(\tilde y, \hat y) + d(\hat y, \hat x) \le 2d(y,x)$ as required.

\end{proof}
\begin{Def}
 Let $Y$ be a locally compact Hausdorff space of finite cohomological dimension over $\Z$. If for some $n>0$, and for all $y\in Y$,  $H_i(Y, Y\setminus \{y\})=0$  for $i \neq n$, and $H_n(Y,Y\setminus \{y\})\cong \Z$ or $0$, then we say that $Y$ is a {\em homology $n$-manifold with boundary}. The boundary of $Y$, $\bd Y$ is defined by $\bd Y = \{ y\in Y:\, H_n(Y, Y\setminus \{y\}) =0 \}$.  If $\bd Y=\emptyset$ then we say $Y$ is a {\em homology $n$-manifold}.
\end{Def}
We need the following theorem of Mitchell:
\begin{Thm}\cite{MIT}\label{MIT}
    If $Y$ is a first countable homology $n$-manifold with boundary then either $\bd Y = \emptyset$ or $\bd Y$ is a homology $(n-1)$-manifold.
\end{Thm}

\begin{Thm}\label{T:homol}
    If a metric $n$-manifold $X$ has a convex geodesic bicombing, then for any $p \in X$ and $r>0$
    $S(p,r)$ is a homology $(n-1)$-manifold.
\end{Thm}
\begin{proof}
We first show that $\bar B(p,r)$ is a homology $n$-manifold with boundary $S(p,r)$. Since $X$ is an $n$-dimensional metrizable manifold and $\bar B(p,r)$ is closed in $X$, the covering dimension of $\bar B(p,r)$ is at most $n$; in particular it has finite cohomological dimension over $\Z$. If $y \in S(p,r)$ then by Lemma \ref{C:cont}, $\bar B(p,r) \setminus \{y\}$ is contractible, as is $\bar B(p,r)$.  
Thus $H_i (\bar B(p,r), \bar B(p,r)\setminus \{y\}) = 0$ for all $i$ by the long exact sequence of a pair. 

For $y \in B(p,r)$, by excision $$H_i(\bar B(p,r),\bar B(p,r)\setminus\{y\})\cong H_i(X, X\setminus \{y\}).$$  Since $X$ is an $n$-manifold, it is a homology $n$-manifold. Hence $H_i(\bar B(p,r),\bar B(p,r)\setminus\{y\})=0$ for $i \neq n$, while this group is isomorphic to $\Z$ when $i=n$.  We have shown that $\bar B(p,r)$  is a homology $n$-manifold with boundary $S(p,r)$, and so by Theorem \ref{MIT}  $S(p,r)$ is a homology $(n-1)$
manifold.  \end{proof}
\begin{Cor} \label{C:sphere}
Let $X$ be a 3-manifold with a proper metric. If $X$ has a convex geodesic bicombing, then $S(p,r)$ is a 2-sphere for each $p\in X$ and $r>0$.
\end{Cor}
\begin{proof}
By Theorem \ref{T:homol}, $S(p,r)$ is a homology 2-manifold. By a theorem of Moore \cite{WILDER}, every homology 2-manifold is a 2-manifold. Thus $S(p,r)$ is a compact 2-manifold.

By the proof of Theorem \ref{T:homol}, $\bar B(p,r)$ is a compact contractible homology 3-manifold with boundary $S(p,r)$. Since $\bar B(p,r)$ is contractible, it is orientable as a homology manifold. Poincar\'e--Lefschetz duality for homology manifolds (see, for example, \cite[Chapter~VIII]{WILDER}) gives
\[
H_i(\bar B(p,r),S(p,r))\cong H^{3-i}(\bar B(p,r)).
\]
Consequently the long exact sequence of the pair $(\bar B(p,r),S(p,r))$ shows that $S(p,r)$ has the same homology as $S^2$. It follows from the classification of compact surfaces that $S(p,r)$ is homeomorphic to $S^2$.
\end{proof}
\begin{Lem}
For $X$ a metric space with a convex geodesic bicombing, fix $p\in X$ and $0<s<r$. The contraction $\tau_p$ induces a homotopy of $S(p,r)$ into $S(p,s)$. If the metric is proper and $X$ is a 3-manifold, the induced map
\[
f:S(p,r)\longrightarrow S(p,s),\qquad f(x)=\tau_p\left(x,\frac{s}{r}\right),
\]
has degree $\pm1$ and in particular is onto.
\end{Lem}
\begin{proof}
Since each geodesic $\sigma_{pq}$ has constant speed, for $x\in S(p,r)$ one has $\tau_p(x,s/r)\in S(p,s)$. Thus the restriction of $\tau_p$ to $S(p,r)\times[s/r,1]$ gives the required homotopy.

Suppose now that the metric is proper and $X$ is a 3-manifold. For $u>0$, put
\[
B_u=\bar B(p,u),\qquad S_u=S(p,u),
\]
and let $i_u:S_u\hookrightarrow X\setminus\{p\}$ be the inclusion. By Corollary \ref{C:sphere}, $S_u$ is a 2-sphere. Poincar\'e--Lefschetz duality and the contractibility of $B_u$ imply that the boundary homomorphism
\[
H_3(B_u,S_u)\longrightarrow H_2(S_u)
\]
is an isomorphism. On the other hand, since $X$ is contractible and is a 3-manifold, the boundary homomorphism
\[
H_3(X,X\setminus\{p\})\longrightarrow H_2(X\setminus\{p\})
\]
is also an isomorphism. Consider the maps of pairs
\[
(B_u,S_u)\longrightarrow (B_u,B_u\setminus\{p\})
\longrightarrow (X,X\setminus\{p\}).
\]
The first map sends the relative fundamental class of the orientable homology manifold $(B_u,S_u)$ to the local orientation class at $p$, and therefore induces an isomorphism on $H_3$; this is the usual localization property of the fundamental class, see \cite[Chapter~VIII]{WILDER}. Since $p$ lies in the interior $B(p,u)$ of $B_u$, the second map induces the standard excision isomorphism on local homology. Thus the composite
\[
H_3(B_u,S_u)\longrightarrow H_3(X,X\setminus\{p\})
\]
is an isomorphism. By naturality of the boundary homomorphism, it follows that
\[
(i_u)_*:H_2(S_u)\longrightarrow H_2(X\setminus\{p\})
\]
is an isomorphism.

The radial homotopy from $S(p,r)$ to $S(p,s)$ takes place in $X\setminus\{p\}$. Hence
\[
i_r\simeq i_s\circ f,
\]
and therefore
\[
(i_r)_*=(i_s)_*\circ f_*.
\]
Since both $(i_r)_*$ and $(i_s)_*$ are isomorphisms, so is
\[
f_*:H_2(S(p,r))\longrightarrow H_2(S(p,s)).
\]
Thus $f$ has degree $\pm1$, and therefore is onto.
\end{proof}

\begin{Thm}
    Let $X$ be a proper metric space which admits a convex geodesic bicombing $\sigma$, and $p \in X$.  Suppose that for some $0<s<r$ the contraction $\tau_p$ induces a homotopy of $S(p,r)$ {\bf onto} $S(p,s)$.  Let $S =S(p,s)$ with the subspace topology, and $\hat S=S(p,s)$ with the induced path metric.  Then the identity function is a homeomorphism from $S$ to $\hat S$.
\end{Thm}
\begin{proof}
Let $d_\ell$ be the path metric on $S(p,s)$. For each $a \in S(p,s)$, let $B_\ell(a,t)=\{b\in S(p,s):d_\ell(a,b)<t\}$ be the ball of radius $t$ in the path metric on $S(p,s)$.
Clearly $B_\ell(a,t) \subset B(a,t)$ in $S$. It remains only to show that, for each $a\in S(p,s)$ and each $t>0$, the point $a$ is an interior point of $B_\ell(a,t)$ in $S$.

Suppose not. Then there is $a \in S(p,s)$ and $t>0$ and a sequence $(a_k) \subset S \setminus B_\ell(a,t)$ with $a_k \to a$ in $S$. Let $f:S(p,r) \to S(p,s)$ be the surjective map defined by $\tau_p$, and define the radial map
\[
\rho:X\setminus B(p,s)\longrightarrow S(p,s),\qquad
\rho(x)=\tau_p\left(x,\frac{s}{d(p,x)}\right).
\]
Thus $f=\rho|_{S(p,r)}$. For each $k$, choose $b_k \in S(p,r)$ with $f(b_k) =a_k$.  By compactness of $S(p,r)$, passing to a subsequence we may assume that $b_k \to b\in S(p,r)$.  By continuity of $f$, $f(b) =a$.  For $u < \min\{ \frac t 2, r-s\}$,  the metric ball of $X$ $B(b, u) \subset X\setminus \bar B(p,s)$.  If $x\in B(b,u)$, then the geodesic $\sigma_{bx}$ has length $d(b,x)<u$ and is contained in $X\setminus\bar B(p,s)$, since $d(p,b)=r$ and $u<r-s$. By Lemma \ref{L:lip}, the radial map $\rho$ sends $\sigma_{bx}$ to a path in $S(p,s)$ of length at most $2d(b,x)$. Hence
\[
d_\ell(a,\rho(x))=d_\ell(\rho(b),\rho(x))\leq 2d(b,x)<2u.
\]
Thus $\rho(B(b,u))\subset B_\ell(a,2u)$.  For all sufficiently large $k$ one has $b_k \in B(b,u)$ and so $a_k=f(b_k)=\rho(b_k)\in B_\ell(a,2u) \subset B_\ell(a,t)$, which is a contradiction.  
We have shown that $a$ is an interior point of $B_\ell(a,t)$ in $S$, and it now follows that the identity function is a homeomorphism.  

\end{proof}
The following corollary follows directly from the previous two results.
\begin{Cor} \label{C:gsph}
If $X$ is a 3-manifold with a proper metric admitting a convex geodesic bicombing, then any metric sphere $S$ in $X$ is a 2-sphere and the path metric on $S$ induces its original topology. In particular, with the path metric, $S$ is a compact length space and hence a geodesic metric space by the Hopf--Rinow theorem for length spaces; see, for example, \cite[Chapter~I.3]{BRI-HAE}.
\end{Cor}

\section{Assouad--Nagata dimension of combable spaces}

\begin{Def}

Let $(X,d)$ be a metric space.

A collection $\mathcal U$ of subsets of $X$ is called $M$-bounded if
\[
\operatorname{diam}(U)\leq M
\]
for every $U\in\mathcal U$.

For $s>0$, a collection $\mathcal U$ is called $s$-disjoint if
\[
d(U,V):=\inf\{d(x,y):x\in U,\ y\in V\}\geq s
\]
whenever $U,V\in\mathcal U$ are distinct.

More generally, for an integer $m\geq 1$, a collection $\mathcal U$ is
called $(m,s)$-disjoint if it can be written as
\[
\mathcal U=\bigcup_{i=1}^{m}\mathcal U^i,
\]
where each subcollection $\mathcal U^i$ is $s$-disjoint.

We will often refer to the subcollections
\[
\mathcal U^1,\ldots,\mathcal U^m
\]
as the \emph{colours} of the cover, and say that a set
$U\in\mathcal U^i$ has colour $i$.

The \emph{Assouad--Nagata dimension} of $X$, denoted by
$\dim_{\mathrm{AN}}(X)$, is the least integer $n\geq 0$ for which there
exists a constant $c>0$ such that, for every $s>0$, the space $X$
admits an $(n+1,s)$-disjoint, $cs$-bounded cover.

The \emph{asymptotic Assouad--Nagata dimension} of $X$, denoted by
$\operatorname{asdim}_{\mathrm{AN}}(X)$, is the least integer $n\geq 0$
for which there exist constants $c>0$ and $s_0>0$ such that, for every
$s\geq s_0$, the space $X$ admits an $(n+1,s)$-disjoint,
$cs$-bounded cover.

Equivalently, $\operatorname{asdim}_{\mathrm{AN}}(X)\leq n$ if and
only if there exist constants $c,b\geq 0$ such that, for every $s>0$,
the space $X$ admits an $(n+1,s)$-disjoint, $(cs+b)$-bounded cover.
\end{Def}

For $p\in X$ and $r,s>0$, put
\[
A(p;r,s):=\{x\in X:r\leq d(p,x)<r+s\}.
\]

We will need the following result, adapted from \cite{JL}.

\begin{Thm}
    
\label{thm:annulus}
Let $(X,d)$ be a metric space and let $p\in X$. Suppose that there
exist an integer $n\geq 1$ and a constant $c>0$ such that, for every
$r\geq 0$ and every $s>0$, the annulus
\[
A(p;r,s)
\]
admits an $(n+1,s)$-disjoint, $cs$-bounded cover. Then
\[
\dim_{\mathrm{AN}}(X)\leq n+1.
\]
More precisely, there exists a constant $C=C(n,c)$ such that, for
every $s>0$, the space $X$ admits an $(n+2,s)$-disjoint,
$Cs$-bounded cover.

If the same hypothesis holds only for $s\geq s_0$, for some $s_0>0$,
then
\[
\operatorname{asdim}_{\mathrm{AN}}(X)\leq n+1.
\]
\end{Thm}

\begin{proof}
We give the proof simultaneously in the two cases. Thus, fix $s>0$ in
the Assouad--Nagata case, and fix $s\geq s_0$ in the asymptotic case.
Put
\[
f(x):=d(p,x).
\]
Thus $f:X\to[0,\infty)$ is $1$-Lipschitz.

We first recall the following elementary covering observation. Suppose
that a metric space $Y$ admits an $(n+1,a)$-disjoint,
$D$-bounded cover, where $n\geq 1$. Then $Y$ admits an
\[
\left(n+2,\frac a3\right)\text{-disjoint},
\qquad
\left(D+\frac{2a}{3}\right)\text{-bounded}
\]
cover such that every point of $Y$ belongs to sets of at least two
different colours.

Indeed, let
\[
\mathcal C=\bigcup_{i=1}^{n+1}\mathcal C^i
\]
be the original coloured cover. For each $C\in\mathcal C$, let $C'$
be its closed $a/3$-neighbourhood. The families
\[
\mathcal C'^{\,i}:=\{C':C\in\mathcal C^i\},
\qquad i=1,\ldots,n+1,
\]
are $a/3$-disjoint, and their members have diameter at most
$D+2a/3$. For each $B\in\mathcal C^j$, define
\[
B^0:=B\setminus
\bigcup_{\substack{1\leq i\leq n+1\\ i\neq j}}
\ \bigcup_{C\in\mathcal C^i} C'.
\]
The sets $B^0$ form an additional colour. The resulting cover has the
asserted boundedness and disjointness properties. Moreover, every
point belongs either to two enlarged sets of distinct colours, or to
an enlarged set and to a set of the additional colour.

Put
\[
t:=(n+2)s.
\]
For $k\in\mathbb N_0$, let
\[
I_k=[kt,(k+1)t).
\]

We first cover the inverse images of the intervals $I_k$ with $k$ odd.
By hypothesis, $f^{-1}(I_k)$ admits an $(n+1,t)$-disjoint,
$ct$-bounded cover. Applying the observation above, one obtains a cover
\[
\mathcal C_k=\bigcup_{i=1}^{n+2}\mathcal C_k^i
\]
of $f^{-1}(I_k)$ which is $(n+2,t/3)$-disjoint and
$(c+1)t$-bounded, and such that every point belongs to sets of at
least two different colours.

For $i=1,\ldots,n+2$, put
\[
J_{k,i}=[kt+(i-1)s,kt+is),
\]
and let $\mathcal B_k^i$ be the collection of connected components of
\[
I_k\setminus J_{k,i}.
\]
Each family $\mathcal B_k^i$ is $s$-disjoint, and every point of
$I_k$ belongs to members of at least $n+1$ of the families
$\mathcal B_k^i$.

Define
\[
\mathcal D_k^i
:=
\left\{
C\cap f^{-1}(B):
C\in\mathcal C_k^i,\ B\in\mathcal B_k^i,\
C\cap f^{-1}(B)\neq\varnothing
\right\}.
\]
The collection
\[
\mathcal D_k:=\bigcup_{i=1}^{n+2}\mathcal D_k^i
\]
covers $f^{-1}(I_k)$. Indeed, every point belongs to two colours of
$\mathcal C_k$ and to at least $n+1$ colours of the families
$\mathcal B_k^i$, so at least one colour occurs in both collections.

Since $t/3\geq s$ and $f$ is $1$-Lipschitz, each
$\mathcal D_k^i$ is $s$-disjoint. Moreover, every member of
$\mathcal D_k$ has diameter at most
\[
(c+1)t=(n+2)(c+1)s.
\]
Set
\[
A:=(n+2)(c+1),
\]
and, for $i=1,\ldots,n+2$, define
\[
\mathcal D^i:=
\bigcup_{\substack{k\geq0\\ k\text{ odd}}}\mathcal D_k^i.
\]
Since distinct odd intervals are separated by an interval of length
$t$, each family $\mathcal D^i$ is still $s$-disjoint. Thus the
union of the $\mathcal D^i$ covers the inverse images of all odd
intervals, and its members have diameter at most $As$.

We now cover the inverse images of the even intervals. Put
\[
\sigma:=(A+2)s.
\]
Notice that $\sigma\geq t$. In the asymptotic case we also have
$\sigma\geq s_0$. For every even $k\geq0$, the hypothesis, applied
with scale $\sigma$, gives an $(n+1,\sigma)$-disjoint,
$c\sigma$-bounded cover of
\[
f^{-1}([kt,kt+\sigma)).
\]
Intersecting its members with $f^{-1}(I_k)$ gives a cover
\[
\mathcal E_k=\bigcup_{i=1}^{n+1}\mathcal E_k^i
\]
of $f^{-1}(I_k)$ which is $(n+1,\sigma)$-disjoint and
$c\sigma$-bounded. Put
\[
\mathcal E^i:=
\bigcup_{\substack{k\geq0\\ k\text{ even}}}\mathcal E_k^i,
\qquad i=1,\ldots,n+1.
\]

For $E\in\mathcal E^i$, define
\[
E^*
:=
E\cup
\bigcup\{D\in\mathcal D^i:d(D,E)<s\}.
\]
Every point of $E^*$ lies at distance at most $As+s$ from $E$.
Consequently,
\[
\operatorname{diam}(E^*)
\leq
c\sigma+2(A+1)s
=
\bigl(c(A+2)+2A+2\bigr)s.
\]

We claim that, for fixed $i$, the sets $E^*$ are pairwise
$s$-disjoint. First suppose that $E,F\in\mathcal E_k^i$ belong to
the same even interval. Then
\[
d(E,F)\geq\sigma=(A+2)s.
\]
No set $D\in\mathcal D^i$ can satisfy both
$d(D,E)<s$ and $d(D,F)<s$, since otherwise
\[
d(E,F)
<
2s+\operatorname{diam}(D)
\leq
(A+2)s=\sigma,
\]
a contradiction. Since $\mathcal D^i$ is $s$-disjoint, it follows
that
\[
d(E^*,F^*)\geq s.
\]

Now suppose that $E$ and $F$ belong to distinct even intervals.
These intervals are separated by at least one odd interval, and hence
\[
d(E,F)\geq t\geq s.
\]
It remains to observe that no member of $\mathcal D^i$ can be within
distance less than $s$ of both $E$ and $F$. Indeed, if
$D\in\mathcal D_k^i$, then $f(D)$ is contained in one connected
component of
\[
I_k\setminus J_{k,i}.
\]
Every such component has distance at least $s$ from at least one of
the two endpoints of $I_k$. Since $f$ is $1$-Lipschitz, $D$ cannot
be within distance less than $s$ of both even intervals adjacent to
$I_k$. The claim follows.

For $i=1,\ldots,n+1$, let $\mathcal U^i$ consist of all sets $E^*$,
with $E\in\mathcal E^i$, together with those members of
$\mathcal D^i$ which were not adjoined to any such $E$. Finally, put
\[
\mathcal U^{n+2}:=\mathcal D^{n+2}.
\]
Then
\[
\mathcal U:=\bigcup_{i=1}^{n+2}\mathcal U^i
\]
is an $(n+2,s)$-disjoint cover of $X$. Every member of
$\mathcal U$ has diameter at most $Cs$, where one may take
\[
C=
\max\left\{
A,\,
c(A+2)+2A+2
\right\},
\qquad
A=(n+2)(c+1).
\]

In the first case the construction works for every $s>0$, and hence
$\dim_{\mathrm{AN}}(X)\leq n+1$. In the second case it works for
every $s\geq s_0$, and hence
$\operatorname{asdim}_{\mathrm{AN}}(X)\leq n+1$.
\end{proof}

The following lemma is essentially proved in \cite{FP} and \cite{JL}
(see also \cite{BBEGLPS}).

\begin{Lem}\label{lem:sphere-nagata}
There exists a universal constant $c_0>0$ with the following
property. If $(S,d)$ is a geodesic metric space homeomorphic to
$S^2$, then, for every $s>0$, $S$ admits a $(3,s)$-disjoint,
$c_0s$-bounded cover. In particular,
\[
\dim_{\mathrm{AN}}(S)\leq2,
\]
with a constant independent of the metric on $S$.
\end{Lem}

\begin{proof}
The only place where planarity enters the proofs in \cite{FP} and
\cite{JL} is the separation argument used in \cite[Lemma~4.1]{FP}.
We explain why the same argument holds on a geodesic metric
$2$-sphere.

In \cite[Lemma~4.1]{FP} one considers an embedded theta curve
contained in a metric annulus. If its three arcs are denoted by
$\alpha,\beta,\gamma$, the arcs may be labelled so that the simple
closed curve $\alpha\cup\gamma$ separates the base point from the
interior of $\beta$. A geodesic from the base point to a suitable
point of $\beta$ must then meet $\alpha\cup\gamma$, which gives the
required contradiction.

The same argument applies on $S^2$. Indeed, an embedded theta curve
in $S^2$ has three complementary regions. If $e$ is the base point,
let $U$ be the complementary region containing $e$. The boundary of
$U$ is the union of two of the three arcs of the theta curve; denote
these by $\alpha$ and $\gamma$, and denote the third arc by $\beta$.
By the Jordan curve theorem, the simple closed curve
$\alpha\cup\gamma$ separates $e$ from the interior of $\beta$.
Thus the proof of \cite[Lemma~4.1]{FP} is unchanged.

After Lemma~4.1, the proofs of \cite[Lemmas~4.2--4.4]{FP} use only
the conclusion of Lemma~4.1 and metric arguments, and hence apply
without change to a geodesic metric $2$-sphere. Consequently the
annular covering statement used in the proof of
\cite[Theorem~2]{JL} holds, with the same universal constants, for
geodesic metric $2$-spheres as well. The remainder of the proof of
\cite[Theorem~2]{JL} is the Hurewicz-type covering argument and does
not use planarity. It follows that there is a universal constant
$c_0$ such that, for every $s>0$, $S$ admits a $(3,s)$-disjoint,
$c_0s$-bounded cover.
\end{proof}

\begin{Thm}\label{thm:AN-combing}
Let $X$ be a $3$-manifold with a proper metric. If $X$ admits a
convex geodesic bicombing, then
\[
\dim_{\mathrm{AN}}(X)\leq3.
\]
In particular,
\[
\operatorname{asdim}_{\mathrm{AN}}(X)\leq3.
\]
\end{Thm}

\begin{proof}
Fix $p\in X$. By Corollary~\ref{C:gsph}, for every $R>0$ the
sphere $S(p,R)$, equipped with its induced path metric, is a compact
geodesic metric space homeomorphic to the $2$-sphere.

By Lemma~\ref{lem:sphere-nagata}, there is a universal constant
$c_0>0$ such that, for every $R>0$ and every $a>0$, the sphere
$S(p,R)$, equipped with its path metric, admits a $(3,a)$-disjoint,
$c_0a$-bounded cover.

We now verify the hypothesis of Theorem~\ref{thm:annulus} with
$n=2$. Fix $r\geq0$ and $s>0$, and consider
\[
A':=A(p;r,s)
=
\{x\in X:r\leq d(p,x)<r+s\}.
\]

If $r\leq s/2$, then
\[
A'\subset B(p,r+s)\subset B(p,3s/2),
\]
and therefore
\[
\operatorname{diam}(A')\leq3s.
\]
Thus in this case the single set $A'$ gives the required cover.

Suppose now that $r>s/2$, and put
\[
R:=r-\frac{s}{2}>0.
\]
Let
\[
S_R:=S(p,R),
\]
equipped with its induced path metric $d_R$. By Lemma~\ref{lem:sphere-nagata},
$S_R$ admits a $(3,2s)$-disjoint,
$2c_0s$-bounded cover
\[
\mathcal B=\bigcup_{i=1}^3\mathcal B^i.
\]

Consider the enlarged annulus
\[
\widetilde A
:=
\left\{
x\in X:
r-\frac{s}{2}\leq d(p,x)<r+s
\right\},
\]
and equip $\widetilde A$ with its induced path metric
$d_{\widetilde A}$.

The contraction $\tau_p$ defines the radial map
\[
\rho:\widetilde A\longrightarrow S_R,
\qquad
\rho(x)
=
\tau_p\left(x,\frac{R}{d(p,x)}\right).
\]
By Lemma~\ref{L:lip}, $\rho$ is $2$-Lipschitz with respect to the ambient
metrics. It follows that
\[
d_R(\rho(x),\rho(y))
\leq
2d_{\widetilde A}(x,y)
\]
for all $x,y\in\widetilde A$. Indeed, if $\gamma$ is a rectifiable
path in $\widetilde A$ joining $x$ to $y$, then $\rho\circ\gamma$ is
a path in $S_R$ and
\[
\operatorname{length}(\rho\circ\gamma)
\leq
2\operatorname{length}(\gamma).
\]

For $B\in\mathcal B$, put
\[
W_B:=\rho^{-1}(B).
\]
The sets $W_B$ form a $(3,s)$-disjoint cover of
$(\widetilde A,d_{\widetilde A})$. Indeed, if $B,B'$ are distinct
members of the same colour, then
\[
d_R(B,B')\geq2s,
\]
and hence
\[
d_{\widetilde A}(W_B,W_{B'})\geq s.
\]

This cover is also uniformly bounded in the path metric. If
$x,y\in W_B$, join $x$ radially to $\rho(x)$, then join
$\rho(x)$ to $\rho(y)$ inside $S_R$, and finally join $\rho(y)$
radially to $y$. Each radial segment has length less than
\[
(r+s)-\left(r-\frac{s}{2}\right)=\frac{3s}{2},
\]
while
\[
d_R(\rho(x),\rho(y))
\leq
2c_0s.
\]
Therefore
\[
d_{\widetilde A}(x,y)
\leq
(2c_0+3)s.
\]
Thus the cover $\{W_B:B\in\mathcal B\}$ is
$(2c_0+3)s$-bounded in $d_{\widetilde A}$.

We restrict this cover to the middle annulus $A'$. The boundedness
with respect to the ambient metric $d$ is immediate, since
\[
d(x,y)\leq d_{\widetilde A}(x,y).
\]
The only point which requires care is the $s$-disjointness, since in
general the path metric of an annulus may be much larger than the
ambient metric.

We claim that if $x,y\in A'$ and
\[
d(x,y)<s,
\]
then
\[
d_{\widetilde A}(x,y)=d(x,y).
\]
Let
\[
\gamma(t)=\sigma_{xy}(t),
\qquad 0\leq t\leq1,
\]
be the geodesic supplied by the bicombing. By convexity of the
bicombing, applied to $\sigma_{xy}$ and the constant geodesic at $p$,
\[
d(p,\gamma(t))
\leq
(1-t)d(p,x)+td(p,y)
<
r+s.
\]
Thus $\gamma$ does not leave $\widetilde A$ through its outer
boundary.

Suppose that for some point $z$ of $\gamma$ one had
\[
d(p,z)<r-\frac{s}{2}.
\]
Since $x,y\in A'$, we have $d(p,x),d(p,y)\geq r$, and hence
\[
d(x,z)
\geq
d(p,x)-d(p,z)
>
\frac{s}{2},
\]
and similarly
\[
d(z,y)>\frac{s}{2}.
\]
But $z$ lies on the geodesic from $x$ to $y$, so
\[
d(x,y)=d(x,z)+d(z,y)>s,
\]
a contradiction. Therefore
\[
r-\frac{s}{2}
\leq
d(p,\gamma(t))
<
r+s
\]
for every $t$, and hence $\gamma\subset\widetilde A$. It follows that
\[
d_{\widetilde A}(x,y)=d(x,y),
\]
as claimed.

Now suppose that two distinct members of the same colour of the
restricted cover of $A'$ were at ambient distance less than $s$.
There would then be points $x,y$ in these two sets with
$d(x,y)<s$. By the claim,
\[
d_{\widetilde A}(x,y)=d(x,y)<s,
\]
contradicting the $s$-disjointness of the cover in
$(\widetilde A,d_{\widetilde A})$.

We have therefore shown that, for every $r\geq0$ and every $s>0$,
the annulus $A(p;r,s)$ admits a $(3,s)$-disjoint,
$Cs$-bounded cover, where
\[
C=\max\{3,2c_0+3\}
\]
is independent of $r$ and $s$. Theorem~\ref{thm:annulus}, with
$n=2$, now gives
\[
\dim_{\mathrm{AN}}(X)\leq3.
\]
The inequality
\[
\operatorname{asdim}_{\mathrm{AN}}(X)\leq3
\]
follows immediately.
\end{proof}

We recall that ${\mathbf I}_{k,\mathrm c}(X)$ denotes the group of compactly supported $k$-dimensional metric integral currents in $X$; $\partial$ denotes the boundary operator and ${\mathbf M}$ the mass.

\begin{Cor}\label{cor:isoperimetric}
Let $X$ be a proper CAT$(0)$ $3$-manifold of asymptotic rank at most
$2$. Then, for every $k\geq2$, there exists a constant $C=C(X,k)$
such that every cycle
\[
T\in {\mathbf I}_{k,\mathrm c}(X)
\]
admits a filling $V\in {\mathbf I}_{k+1,\mathrm c}(X)$ with
$\partial V=T$ and
\[
{\mathbf M}(V)\leq C\,{\mathbf M}(T).
\]
\end{Cor}

\begin{proof}
By Theorem~\ref{thm:AN-combing}, $X$ has Assouad--Nagata dimension at
most $3$, and hence finite asymptotic Assouad--Nagata dimension. Since
$X$ is proper and CAT$(0)$, it is a Hadamard space. The result now
follows from \cite{Pet}.
\end{proof}

\end{document}